\documentclass[a4paper,10pt]{amsart}
\usepackage[english]{babel}
\usepackage[active]{srcltx}
\usepackage{amsmath,amssymb,amsthm}
\usepackage{multicol}

\usepackage[pdftex]{color}
\usepackage[bookmarks=true,hyperindex,pdftex,colorlinks,citecolor=blue,linkcolor=blue,urlcolor=blue]
{hyperref}

\usepackage{graphicx}

\usepackage{pgf,tikz}
\usetikzlibrary{arrows}

\theoremstyle{plain}
\newtheorem{theorem}{Theorem}[section]
\newtheorem{corollary}[theorem]{Corollary}
\newtheorem{proposition}[theorem]{Proposition}
\newtheorem{lemma}[theorem]{Lemma}

\theoremstyle{definition}

\newtheorem{example}[theorem]{Example}

 \DeclareMathOperator{\re}{Re\,}
 
 \DeclareMathOperator{\NRA}{NRA}

\newcommand{\R}{\mathbb{R}}
\newcommand{\N}{\mathbb{N}}

\newcommand{\eps}{\varepsilon}

\renewcommand{\leq}{\leqslant}
\renewcommand{\geq}{\geqslant}

\begin{document}
\title[Numerical radius attaining compact operators]{Numerical radius attaining compact linear operators}

\dedicatory{Dedicated to Richard Aron on the occasion of his retirement from Kent State University}

\author[Capel]{\'{A}ngela Capel}
\author[Mart\'{\i}n]{Miguel Mart\'{\i}n}
\author[Mer\'{\i}]{Javier Mer\'{\i}}

\address[Capel]{Instituto de Ciencias Matem\'{a}ticas CSIC-UAM-UC3M-UCM \\ C/ Nicol\'{a}s Cabrera 13-15 \\ 28049 Madrid \\ Spain
\newline
\href{http://orcid.org/0000-0001-6713-6760}{ORCID: \texttt{0000-0001-6713-6760} }
 }
\email{angela.capel@icmat.es}

\address[Mart\'{\i}n]{Departamento de An\'{a}lisis Matem\'{a}tico \\ Facultad de
 Ciencias \\ Universidad de Granada \\ 18071 Granada\\ Spain
\newline
\href{http://orcid.org/0000-0003-4502-798X}{ORCID: \texttt{0000-0003-4502-798X} }
 }
\email{mmartins@ugr.es}

\address[Mer\'{\i}]{Departamento de An\'{a}lisis Matem\'{a}tico \\ Facultad de
 Ciencias \\ Universidad de Granada \\ 18071 Granada\\ Spain
\newline
\href{http://orcid.org/0000-0002-0625-5552}{ORCID: \texttt{0000-0002-0625-5552} }
 }
\email{jmeri@ugr.es}

\begin{abstract}
We show that there are compact linear operators on Banach spaces which cannot be approximated by numerical radius attaining operators.
\end{abstract}

\thanks{First author partially supported by a La Caixa-Severo Ochoa grant. Second and third authors partially supported by Spanish MINECO and FEDER project no.\ MTM2012-31755 and by Junta de Andaluc\'{\i}a and FEDER grant FQM-185.}


\maketitle

\section{Introduction}

Given a real or complex Banach space $X$, we write $S_X$ and $B_X$ to denote its unit sphere and its unit ball, respectively, and $X^*$ for the topological dual space of $X$. If $Y$ is another Banach space, $L(X,Y)$ denotes the space of all bounded linear operators from $X$ to $Y$ and we just write $L(X)$ for $L(X,X)$. The space of all compact linear operators on $X$ will be denoted by $K(X)$. We consider the set
$$
\Pi(X)=\bigl\{(x,x^*)\in X\times X^*\,:\, x\in S_{X},\,x^*\in S_{X^*},\, x^*(x)=1\bigr\}.
$$
The \emph{numerical range} of $T\in L(X)$ is the subset of the base field given by
$$
V(T)=\bigl\{x^*(Tx)\,:\, (x,x^*)\in \Pi(X)\bigr\}.
$$
A complete survey on numerical ranges and their relations to spectral theory of operators can be found in the books by F.~Bonsall and J.~Duncan \cite{B-D1,B-D2}, where we refer the reader for general information and background. The recent development of this topic can be found in \cite{Cabrera-Rodriguez} and references therein.

The \emph{numerical radius} of $T\in L(X)$ is given by
$$
v(T)=\sup\bigl\{|\lambda|\,:\,\lambda\in V(T)\bigr\}.
$$
It is clear that $v$ is a continuous seminorm which satisfies $v(T)\leq \|T\|$ for every $T\in L(X)$. It is said that $T$ \emph{attains its numerical radius} when the supremum defining $v(T)$ is actually a maximum. We will denote by $\NRA(X)$ the set of numerical radius attaining operators on $X$. One clearly has that $\NRA(X)=L(X)$ if $X$ is finite-dimensional. Even in a separable Hilbert space it is not difficult to find diagonal operators which do not attain their numerical radii. Our paper deals with the study of the density of $\NRA(X)$. This study was started in the PhD dissertation of B.~Sims of 1972 (see \cite{BS}), parallel to the study of norm attaining operators initiated by J.~Lindenstrauss in 1963 \cite{Lindens} (recall that a bounded linear operator $T$ between two Banach spaces $X$ and $Y$ is said to \emph{attain its norm} if there is $x\in S_X$ such that $\|Tx\|=\|T\|$). Among the positive results on this topic, we would like to mention that the set of numerical radius attaining operators is dense for Banach spaces with the Radon-Nikod\'{y}m property (M.~Acosta and R.~Pay\'{a} \cite{AP-RNP}) and for $L_1(\mu)$ spaces (M.~Acosta \cite{Acosta-CL}) and $C(K)$ spaces (C.~Cardassi \cite{Cardassi-PAMS}). On the other hand, the first example of Banach space for which the set of numerical radius attaining operators is not dense was given by R.~Pay\'{a} in 1992 \cite{Paya}. Another counterexample was discovered shortly later by M.~Acosta, F.~Aguirre and R.~Pay\'{a} \cite{AAP}. The proofs in these two papers are tricky and non-trivial, and, in both examples, the operators shown that cannot be approximated by numerical radius attaining operators are not compact.

Our aim in this paper is to show that there are \textbf{compact} linear operators which cannot be approximated by numerical radius attaining operators. The analogous problem about compact operators which cannot be approximated by norm attaining operators has been recently solved by the second author of this manuscript \cite{Mar-JFA}. Actually, the proofs here mix some ideas from that paper with some ideas from the already mentioned counterexamples for the density of numerical radius attaining operators \cite{AAP,Paya}.

Let us comment that the counterexamples in \cite{AAP,Paya} follow similar lines and borrow some ideas from the seminal paper by Lindenstrauss \cite{Lindens}: they are of the form $Y\oplus_\infty Z$, where $Y^*$ is smooth enough, $Z$ fails to have extreme points in its unit ball in a strong way, and there are operators from $Z$ into $Y$ that cannot be approximated by numerical radius attaining operators (norm attaining operators in the counterexample of Lindenstrauss). Moreover, by construction, the operators shown there that cannot be approximated by numerical radius attaining operators are not compact, as they are constructed using operators from $Z$ into $Y$ which are not compact. Actually, the existence of non-compact operators from $Z$ to $Y$ is one of the building blocks of their proofs. Here this fact will be replaced by the use of Banach spaces without the approximation property.

Taking advantage of some recent ideas, it is now easy (up to a non-trivial old result by W.~Schachermayer) to present new examples of (non-compact) operators which cannot be approximated by numerical radius attaining operators. However, as we will show at the end of the paper, these new examples do not work for compact operators (see Example~\ref{example:L1+C-positive-compact}).

\begin{example}\label{example:L1+C-negative}
{\slshape There are bounded linear operators on the real spaces $X=C[0,1]\oplus_1 L_1[0,1]$ and $Y=C[0,1]\oplus_\infty L_1[0,1]$ which cannot be approximated by numerical radius attaining operators.}\newline Indeed, suppose for the sake of contradiction that $\NRA(X)$ is dense in $L(X)$. As $v(T)=\|T\|$ for every $T\in L(X)$ (\cite[Theorem~2.2]{D-Mc-P-W} and \cite[Proposition~1]{M-P}), it follows that norm-attaining operators from $X$ into $X$ are dense in $L(X)$. Then, we may use \cite[Proposition~2.9]{ACKLM} and \cite[Lemma~2]{PayaSaleh} to get that norm attaining operators from $L_1[0,1]$ into $C[0,1]$ are dense in $L(L_1[0,1],C[0,1])$, but this is not the case as shown by W.~Schachermayer \cite{Schachermayer-classical}. The proof for $Y$ is absolutely analogous.
\end{example}

The outline of the paper is as follows. We devote section~\ref{sect-example} to present the promised example of a compact linear operator which cannot be approximated by numerical radius attaining operators. As this example exists, it makes sense to study sufficient conditions on a Banach space $X$ to assure that $\NRA(X)\cap K(X)$ is dense in $K(X)$. In section~\ref{sect-positive} we review some of these conditions which can be easily deduced from previous results in the literature, and present some interesting new examples.

\section{The example}\label{sect-example}

Here is the main result of this paper.

\begin{theorem}\label{theorem}
There is a compact linear operator which cannot be approximated by numerical radius attaining operators.
\end{theorem}

To state our example properly, we need to recall the definition and some basic results about the approximation property. We refer to \cite{Jarchow,Linden-Tz,LTII} for background. A Banach space $X$ has the \emph{approximation property} if for every compact set $K$ and every $\varepsilon>0$, there is $R\in L(X)$ of finite-rank such that $\|x-R(x)\| <\varepsilon$ for all $x\in K$. It was shown by P.~Enflo in 1973 that there are Banach spaces failing the approximation property. Actually, there are closed subspaces of $c_0$ and $\ell_p$ ($p\neq 2$) without the approximation property \cite[Theorem~2.d.6]{Linden-Tz}, \cite[Theorem~1.g.4]{LTII}, and so there are quotients of $\ell_p$ without the approximation property (use \cite[Corollary~18.3.5]{Jarchow}).

Our theorem will be proved if we state the validity of the following family of examples.

\begin{example}\label{example-main-negative}
{\slshape Given $1<p<2$ and a quotient $Y$ of $\ell_p$ without the approximation property, there exists a closed subspace $Z$ of $c_0$ such that $K(Y\oplus_\infty Z)$ is not contained in the closure of $\NRA(Y\oplus_\infty Z)$.}
\end{example}

As we already mentioned, the idea of considering spaces of the form $Y\oplus_\infty Z$, where $Y^*$ is smooth enough and $Z$ fails to have extreme points, to produce counterexamples for numerical radius attaining operators goes back to R.~Pay\'{a} \cite{Paya} and to M.~Acosta, F.~Aguirre and R.~Pay\'{a} \cite{AAP}. The proof of our result actually is an extension of the one given in \cite[\S2]{AAP}, but we necessarily have to change the space $Y$ involved, as Hilbert spaces have the approximation property \cite[Proposition~18.5.4]{Jarchow}, and so do their closed subspaces. This makes life a little more complicated for us.

We divide the proof of the main result into two lemmata and one proposition for the sake of clearness.

The first result deals with the smoothness of the space $Y^*$.
We recall briefly some facts about differentiability of functions on Banach spaces which will be useful to our discussion. We refer the reader to \cite{D-G-Z} for the details. Given Banach spaces $X$, $Y$ and a function $f: X\longrightarrow Y$ which is G\^ateaux (Fr\'echet) differentiable, we write $D_f(x)\in L(X,Y)$ for the differential of $f$ at $x\in X$. We say that the norm $\|\cdot\|$ of $X$ is smooth if it is G\^ateaux differentiable at every $x\in X\setminus \{0\}$. The normalized duality mapping $J_X: X \longrightarrow 2^{X^*}$ of $X$ is given by
$$
J(x)=\{x^*\in X^* \ : \ x^*(x)=\|x^*\|^2=\|x\|^2\} \qquad (x\in X).
$$
If the norm of $X$ is smooth, this mapping is single-valued and the map $\widetilde{J}_X: X\setminus\{0\} \longrightarrow S_{X^*}$ given by
$$
\widetilde{J}_X(x)=J\left(\frac{x}{\|x\|}\right)=\frac{J(x)}{\|J(x)\|} \qquad (x\in X\setminus\{0\})
$$
is well defined. Let us observe that $\widetilde{J}_X(x)$ can be alternatively defined as the unique $x^*\in S_{X^*}$ such that $x^*(x)=\|x\|$. If, moreover, we assume that the norm of $X$ is $C^2$-smooth on $X\setminus \{0\}$, then $\widetilde{J}_X$ is Fr\'echet differentiable on $X\setminus \{0\}$.

\begin{lemma}\label{lemma_on_Y}
Let $Y$ be a Banach space such that the norm of $Y^*$ is $C^2$-smooth on $Y^*\setminus\{0\}$ and let $Z$ be a Banach space. Suppose that $A\in L(Y)$, $B\in L(Z, Y)$, and $(y_0,y_0^*)\in \Pi(Y)$ satisfy that
\begin{equation*}
|y^*(Ay)|+\|B^*y^*\|\leq |y_0^*(Ay_0)|+\|B^*y_0^*\|
\end{equation*}
for all $(y,y^*)\in \Pi(Y)$. Then,
\begin{equation*}
\lim_{t\to 0}\frac{\|B^*y_0^*+tB^*h^*\|+\|B^*y_0^*-tB^*h^*\|-2\|B^*y_0^*\|}{t} = 0
\end{equation*}
for every $h^*\in S_{Y^*}$.
\end{lemma}

\begin{proof}
Observe first that the assumption on $Y$ implies reflexivity \cite[Proposition~II.3.4]{D-G-Z}. Therefore, we may and do identify $Y^{**}$ with $Y$ and consider the normalized duality mapping $\widetilde{J}_{Y^*}: Y^*\setminus\{0\}\longrightarrow S_Y$ and observe that it is Fr\'echet differentiable by the hypothesis on $Y$. Hence, the function $F:Y^*\setminus\{0\}\longrightarrow \R$ given by
$$
F(y^*)=\bigl|y^*\bigl[A\bigl(\widetilde{J}_{Y^*}(y^*)\bigr)\bigr]\bigr| \qquad \big(y^*\in Y^*\setminus\{0\}\big)
$$
is Fr\'echet differentiable at every $y^*\in Y^*\setminus\{0\}$ for which $F(y^*)\neq 0$. Next, we fix $h^*\in S_{Y^*}$ and for $0\leq t<1$ we define:
$$
y_t^*=y_0^*+th^*, \qquad \phi(t)=\|y_t^*\|, \qquad F_1(t)=F(y_t^*), \qquad \text{and} \qquad F_2(t)=\|B^*y_t^*\|.
$$
On the one hand, $F_2$ is right-differentiable at the origin as it is a convex function. On the other hand, if we assume that $0\neq |y_0^*(Ay_0)|=F(y_0^*)=F_1(0)$ (observe that $y_0=\widetilde{J}_{Y^*}(y_0^*)$ by smoothness), we get that $F_1$ is differentiable at the origin.

Now, by using the inequality in the hypothesis for
$$
y^*=\phi(t)^{-1}y_t^* \qquad \text{ and } \qquad  y=\widetilde{J}_{Y^*}(\phi(t)^{-1}y_t^*)= \widetilde{J}_{Y^*}(y_t^*),
$$
we obtain that
\begin{equation}\label{eq:thm-AAP-desigualdad-H1-H2}
F_1(t)+F_2(t)\leq \phi(t)[F_1(0)+F_2(0)] \qquad (0\leq t<1),
\end{equation}
which gives
$$
\frac{F_1(t)- F_1(0)}{t}+\frac{F_2(t)-F_2(0)}{t}\leq \frac{\phi(t)-1}{t}[F_1(0)+F_2(0)] \qquad (0<t<1).
$$
Taking right-limits with $t\rightarrow 0$, we obtain
$$
F'_1(0)+ \partial_+ F_2(0)\leq \phi'(0)[F_1(0)+F_2(0)]
$$
(where $\partial_+F_2(0)$ is the right-derivative of $F_2$ at 0) or, equivalently,
$$
D_F(y_0^*)(h^*)+\lim_{t\to 0^+}\frac{\|B^*y_0^*+tB^*h^*\|-\|B^*y_0^*\|}{t}\leq D_{\|\cdot\|_{Y^*}}(y_0^*)(h^*)\bigl[F(y^*_0) + \|B^*(y_0^*)\|\bigr].
$$
If we repeat the above argument for $-h^*$, we get the analogous inequality
$$
D_F(y_0^*)(-h^*)+\lim_{t\to 0^+}\frac{\|B^*y_0^*-tB^*h^*\|-\|B^*y_0^*\|}{t}\leq D_{\|\cdot\|_{Y^*}}(y_0^*)(-h^*)\bigl[F(y^*_0) + \|B^*(y_0^*)\|\bigr].
$$
Adding the above two equations, taking into account that both $F$ and the norm of $Y^*$ are Fr\'{e}chet differentiable, we obtain
\begin{equation}\label{eq:thm-AAP-key-limit-with+}
\lim_{t\to 0^+}\frac{\|B^*y_0^*+tB^*h^*\|+\|B^*y_0^*-tB^*h^*\|-2\|B^*y_0^*\|}{t} \leq 0
\end{equation}
as desired. We recall that we required that $|y_0^*(Ay_0)|\neq0$ to use the differentiability of $F_1$. If, otherwise, we have $|y_0^*(Ay_0)|=0$, observe that inequality \eqref{eq:thm-AAP-desigualdad-H1-H2} implies
$$
F_2(t)\leq \phi(t)F_2(0),
$$
and we can repeat the arguments above without the use of $F_1$.

Next, observe that the function $\dfrac{\|B^*y_0^*+tB^*h^*\|+\|B^*y_0^*-tB^*h^*\|-2\|B^*y_0^*\|}{t}$ is non-negative for every $t>0$ by the convexity of the norm, and so the limit in \eqref{eq:thm-AAP-key-limit-with+} is actually equal to zero. Finally, as changing $t$ by $-t$ in this limit just changes the sign of the function and the limit is zero, we may replace right-limit by regular limit, getting the statement of the proposition.
\end{proof}

We next deal with the needed condition on the space $Z$, which is a strong way of failing to have extreme points. Let $Z$ be a Banach space. We say that $Z$ is \emph{flat} at a point $z_0\in S_Z$ in the direction $z\in Z$ if $\|z_0\pm z\|\leq 1$ and we write
$$
\mathrm{Flat}(z_0)=\{z\in Z\,:\,\|z_0\pm z\|\leq 1\}
$$
to denote the \emph{set of directions of flatness} at $z_0$.
Observe that $z_0$ is an extreme point of $B_Z$ if and only if $Z$ is not flat at $z_0$ in any direction (i.e.\ $\mathrm{Flat}(z_0)=\{0\}$). We say that $B_Z$ is \emph{strongly flat} if for every $z_0\in S_Z$, the closed linear span of $\mathrm{Flat}(z_0)$ has finite-codimension. Easy examples of spaces of this kind are $c_0$ and all its closed infinite-dimensional subspaces (see \cite[Lemma~2.2]{Mar-Racsam}). This definition is stronger than the one of uniformly asymptotically flat \cite[\S~V]{Godefroy} and, hence, it implies asymptotic uniform smoothness \cite{JohLinPreSch2002}.

\begin{lemma}\label{lemma-on-Z}
Let $Z$ be a strongly flat Banach space and let $Y$ be a Banach space.
Suppose that for $B\in L(Z,Y)$ there is $y_0^*\in S_{Y^*}$ such that
\begin{equation*}
\lim_{t\to 0^+}\frac{\|B^*y_0^*+tB^*h^*\|+\|B^*y_0^*-tB^*h^*\|-2\|B^*y_0^*\|}{t} \leq 0
\end{equation*}
for every $h^*\in S_{Y^*}$ and that $B^*y_0^*$ attains its norm on $Z$. Then, $B$ has finite-rank.
\end{lemma}

\begin{proof}
Write $z_0^*=B^*y_0^*$. As $z_0^*$ attains its norm, we may take $z_0\in S_Z$ such that $\re z_0^*(z_0)=\|z_0^*\|$. We claim that $Bz=0$ for every $z\in \mathrm{Flat}(z_0)$, and this finishes the proof by the hypothesis on $Z$. Therefore, let us prove the claim. Fixed $z\in \mathrm{Flat}(z_0)$, for each $h^*\in S_{Y^*}$, we write
$z^*=\theta B^*h^*$, where $\theta$ is a modulus-one scalar satisfying that $\re z^*(z)=|z^*(z)|$. Next, given $\eps>0$, we use the inequality in the hypothesis to find $r>0$ such that
\begin{equation*}
\|z_0^*+tz^*\|+\|z_0^*-tz^*\|<2\|z_0^*\|+t\eps
\end{equation*}
for every $t\in (0,r)$. Now, as $\|z_0\pm z\|\leq 1$, we get that
\begin{align*}
2\|z_0^*\|+t\eps&>\|z_0^*+tz^*\|+\|z_0-tz^*\|\\
&\geq \re\Big([z_0^*+tz^*](z_0+z)+[z_0^*-tz^*](z_0-z)\Big)\\
&=2\|z_0^*\|+2t\re z^*(z)=2\|z_0^*\|+2t|z^*(z)|.
\end{align*}
This gives that $2|z^*(z)|<\eps$, and the arbitrariness of $\eps$ implies that
$$
0=|z^*(z)|=\bigl|[B^*h^*](z)\bigr|=|h^*(Bz)|.
$$
Since this is true for every $h^*\in S_{Y^*}$, we get that $Bz=0$, as claimed.
\end{proof}

Finally, numerical radius attainment appears in our last preliminary result. We will make use of the following lemma from \cite{Paya}, which we state for the convenience of the reader.

\begin{lemma}\cite[Lemma~1.2]{Paya}\label{lem:Rafa}
Let $Y$, $Z$ be Banach spaces, $X=Y\oplus_\infty Z$ and $P_Y$, $P_Z$ the projections from $X$ onto $Y$ and $Z$, respectively. For $T\in L(X)$, we have
\begin{enumerate}
\item $v(T)=\max\{v(P_YT), v(P_ZT)\}$.
\item If $T\in \NRA(X)$ and $v(P_YT)>v(P_ZT)$, then $P_YT\in \NRA(X)$.
\item $v(P_YT)=\sup\{|y^*(P_YT(y+z))| \, : \, (y,y^*)\in \Pi(Y),\, z\in B_Z \}$ and $P_YT\in \NRA(X)$ if and only if this supremum is attained.
\end{enumerate}
\end{lemma}

Here is the last preliminary result in the way to prove Theorem~\ref{theorem}, where we glue our two lemmata with the fact that an operator of a particular form attains its numerical radius.

\begin{proposition}\label{all-together}
Let $Y$ be a Banach space such that the norm of $Y^*$ is $C^2$-smooth on $Y^*\setminus\{0\}$ and let $Z$ be a strongly flat Banach space. Consider $X=Y\oplus_\infty Z$ and for $A\in L(Y)$ and $B\in L(Z, Y)$, define $T\in L(X)$ by
$$
T(y+z)=A(y)+B(z) \qquad  \bigl(y\in Y,\ z\in Z\bigr).
$$
If $T\in \NRA(X)$, then $B$ is of finite-rank.
\end{proposition}

\begin{proof}
Consider the projection $P_Y$ from $X$ onto $Y$. It is clear that $P_YT=T$ and Lemma~\ref{lem:Rafa}.(3) provides the existence of $(y_0,y_0^*)\in \Pi(Y)$ and $z_0\in B_Z$ such that
$$
|y^*(Ay+Bz)|\leq |y_0^*(Ay_0+Bz_0)|
$$
for every $(y,y^*)\in \Pi(Y)$ and every $z\in B_Z$. By rotating $z$, we actually get
$$
|y^*(Ay)|+|y^*(Bz)|\leq |y_0^*(Ay_0)|+|y_0^*(Bz_0)|
$$
or, equivalently,
\begin{equation}\label{eq:thm-AAP-desigualdad-funcion-maximo}
|y^*(Ay)|+|[B^*y^*](z)|\leq |y_0^*(Ay_0)|+|[B^*y_0^*](z_0)|.
\end{equation}
By taking supremum on $z\in B_Z$, we obtain
\begin{equation*}
|y^*(Ay)|+\|B^*y^*\|\leq |y_0^*(Ay_0)|+\|B^*y_0^*\|
\end{equation*}
for all $(y,y^*)\in \Pi(Y)$. As the norm of $Y^*$ is $C^2$ smooth at $Y^*\setminus \{0\}$, it follows from Lemma~\ref{lemma_on_Y} that
\begin{equation*}
\lim_{t\to 0}\frac{\|B^*y_0^*+tB^*h^*\|+\|B^*y_0^*-tB^*h^*\|-2\|B^*y_0^*\|}{t} = 0
\end{equation*}
for every $h^*\in S_{Y^*}$. On the other hand, when we take $(y,y^*)=(y_0,y_0^*)$ in equation \eqref{eq:thm-AAP-desigualdad-funcion-maximo}, we obtain
$$
\bigl|[B^*y_0^*](z)\bigr|\leq \bigl|[B^*y_0^*](z_0)\bigr|
$$
for every $z\in B_Z$, meaning that the functional $B^*y_0^*\in Z^*$ attains its norm at $z_0$. These two facts and the assumption on $Z$ allow us to apply Lemma~\ref{lemma-on-Z} to get that $B$ is of finite-rank.
\end{proof}

We are finally able to present the proof of the main result of the paper.

\begin{proof}[Proof of Example~\ref{example-main-negative}]
Fix $1<p<2$ and let $Y$ be a quotient of $\ell_p$ without the approximation property (use \cite[Theorem~2.d.6]{Linden-Tz} and \cite[Corollary~18.3.5]{Jarchow} for the existence). Then, there exist a closed subspace $Z$ of $c_0$ and a compact linear operator $S:Z\longrightarrow Y$ which cannot be approximated by finite-rank operators \cite[Theorem~18.3.2]{Jarchow}.
Write $X=Y \oplus_\infty Z$ and let $T\in K(X)$ be defined by
$$
T(y+z)=S(z) \qquad (y\in Y, z\in Z).
$$
Suppose, for the sake of contradiction, that there is a sequence $\{T_n\}$ in $\NRA(X)$ converging to $T$ in norm. We clearly have that $P_YT=T$ and $P_ZT=0$, where $P_Y,P_Z$ are the projections from $X$ onto $Y$ and $Z$, respectively. We get that $\{P_YT_n\}\longrightarrow T$, $\{P_ZT_n\}\longrightarrow 0$, so
$\{v(P_Y T_n)\}\longrightarrow v(T)=v(P_YT)$ and $\{v(P_ZT_n)\}\longrightarrow 0$. It follows from Lemma~\ref{lem:Rafa} that $v(T)=\|S\|>0$ and that $P_YT_n \in \NRA(X)$ for every $n$ large enough. Therefore, removing some terms of the sequence $\{T_n\}$ and replacing $T_n$ by $P_YT_n$, there is no restriction in assuming that
$$
T_n \in \NRA(X), \qquad P_YT_n=T_n  \qquad \forall n\in \N, \qquad \text{and} \qquad \|T_n-T\|\longrightarrow 0.
$$
Now, observe that for each $n\in \N$ there are operators $A_n\in L(Y)$ and $B_n\in L(Z,Y)$ such that
$$
T_n(y+z)=A_n(y)+B_n(z) \qquad (y\in Y,z\in Z).
$$
The norm of $Y^*$ is $C^2$-smooth (indeed, $Y^*$ is a subspace of $\ell_q$ with $q>2$ and then we may use \cite[Theorem~V.1.1]{D-G-Z}) and $Z$ is strongly flat by \cite[Lemma~2.2]{Mar-Racsam}, so we can use Proposition~\ref{all-together} with $T_n$ to conclude that $B_n$ is a finite-rank operator for every $n\in \N$. But this leads to a contradiction because $\{B_n\}$ converges in norm to $S$, finishing thus the proof.
\end{proof}

\section{Some positive results}\label{sect-positive}
Let us finish the paper with a small discussion about positive results on numerical radius attaining compact operators. First, we show that three conditions assuring the density of the set of numerical radius attaining operators also work for the case of compact operators. We need some definitions.

A Banach space $X$ has \emph{property $\alpha$} (respectively \emph{property $\beta$}) if there are two sets $\{ x_i\,:\, i \in I\} \subset S_X$, $\{ x^{*}_i\,:\, i \in I \} \subset S_{X^*}$ and a constant $0 \leq \rho <1$ such that conditions $(i)$, $(ii)$ and $(iii)_\alpha$ (resp.\ $(iii)_\beta$) below hold:
\begin{enumerate}
\item[$(i)$] $x^{*}_i(x_i)=1$, $\forall i \in I $;
\item[$(ii)$] $|x^*_i(x_j)|\leq \rho <1$ if $i, j \in I, i \ne j$;
\item[$(iii)_\alpha$] $B_{X}$ is the absolutely closed convex hull of $\{x_i\,:\, i\in I\}$;
\item[$(iii)_\beta$] $B_{X^*}$ is the absolutely weakly$^*$-closed convex hull of $\{x^*_i\,:\, i\in I\}$.
\end{enumerate}
We refer to \cite{Moreno} and references therein for more information and background. The prototype of space with property $\alpha$ is $\ell_1$. Examples of Banach spaces with property $\beta$ are closed subspaces of $\ell_\infty(I)$ containing the canonical copy of $c_0(I)$.

\begin{proposition}
Let $X$ be a Banach space having at least one of the following properties:
\begin{enumerate}
\item[(a)] $X$ has the Radon-Nikodym property;
\item[(b)] $X$ has property $\alpha$;
\item[(c)] $X$ has property $\beta$.
\end{enumerate}
Then, $\NRA(X)\cap K(X)$ is dense in $K(X)$.
\end{proposition}

\begin{proof}
We first observe that if every operator in $X$ can be perturbed by a compact linear operator of arbitrarily small norm to obtain a numerical radius attaining operator, then $\NRA(X)\cap K(X)$ is dense in $K(X)$. Let us give references to show that this happens in the cases of the proposition.
\begin{enumerate}
\item[(a)] In a Banach space with the Radon-Nikodym property, every operator may be perturbed by a rank-one operator of arbitrarily small norm to obtain an operator which attains its numerical radius \cite[Theorem~2.4]{AP-RNP}.
\item[(b)] It is shown in \cite{Ac-use} that in a Banach space with property $\alpha$, given $T\in L(X)$ and $\varepsilon >0$, there exists a nuclear operator $A$ with nuclear norm less than $\varepsilon$ such that $T+A$ attains its numerical radius.
\item[(c)] From \cite[Theorem~5]{Ac-Beta}, every operator in a space with property $\beta$ may be perturbed by a rank-one operator of arbitrarily small norm to obtain a numerical radius attaining operator.\qedhere
\end{enumerate}
\end{proof}

Our next result deals with the so-called CL-spaces. A Banach space $X$ is a \emph{CL-space} if its unit ball is equal to the absolutely convex hull of every maximal convex set of its unit sphere. Examples of CL-spaces are the real or complex $C(K)$ spaces and the real spaces $L_1(\mu)$. We refer the reader to \cite{MartPaya-CL} and references therein for more information and background. Let $X$ be a CL-space. Then, for every $T\in L(X)$ one has that $v(T)=\|T\|$ and that $T$ attains its norm if and only if $T$ attains its numerical radius \cite{Acosta-CL}. Therefore, the following result is immediate.

\begin{proposition}\label{prop-CL}
Let $X$ be a CL-space. If the set of norm attaining compact operators is dense in $K(X)$, then $\NRA(X)\cap K(X)$ is dense in $K(X)$.
\end{proposition}

As a corollary, using the result by J.~Johnson and J.~Wolfe that every compact linear operator having $C(K)$ or $L_1(\mu)$ as domain or range can be approximated by norm attaining compact operators \cite[Theorems 3 and 4]{JoWo}, we get the following.

\begin{corollary}
Let $X$ be a real or complex $C(K)$ space or a real $L_1(\mu)$ space. Then $\NRA(X)\cap K(X)$ is dense in $K(X)$.
\end{corollary}

We finish the paper with another consequence of Proposition~\ref{prop-CL}. Compare it with Example~\ref{example:L1+C-negative}.

\begin{example}\label{example:L1+C-positive-compact}
{\slshape Compact linear operators on the real spaces
$$
X=C[0,1]\oplus_1 L_1[0,1] \qquad \text{and} \quad Y=C[0,1]\oplus_\infty L_1[0,1]
$$
can be approximated by numerical radius attaining compact operators.}
\end{example}

\begin{proof}
Being a CL-space is stable by finite $\ell_1$-sum \cite[Proposition~9]{MartPaya-CL}, so $X$ is a CL-space. Therefore, by Proposition~\ref{prop-CL}, it is enough to show that norm attaining compact operators on $X$ are dense. Indeed, by \cite[Theorems 3 and 4]{JoWo}, we have that norm attaining compact operators in both $L(C[0,1],X)$ and $L(L_1[0,1],X)$ are dense. Then, norm attaining compact operators on $L(X)$ are dense by \cite[Lemma~3.7]{Mar-Racsam}. A dual argument can be given for $Y$, using \cite[Proposition~8]{MartPaya-CL}, \cite[Theorems 3 and 4]{JoWo} again, and an obvious adaptation to compact operators of \cite[Lemma~2]{PayaSaleh}.
\end{proof}

\end{document}